\documentclass[12pt]{amsart}
\usepackage{amssymb}
\usepackage{amsmath}
\usepackage{amsthm}
\usepackage{times}
\usepackage{graphicx}

\usepackage[curve,all]{xy}
 \xyoption{curve}
 \xyoption{line}
\xyoption{arc}

\usepackage{color}
\usepackage{amsthm}
\newtheorem{theorem}{Theorem}[section]
\newtheorem{lemma}[theorem]{Lemma}
\newtheorem{corollary}[theorem]{Corollary}

\newtheorem{prop}[theorem]{Proposition}

\theoremstyle{definition}

\theoremstyle{remark}
\newtheorem{remark}[theorem]{Remark}

\numberwithin{equation}{section}

\newcommand{\calE}{\mathcal{E}}

\newcommand{\calO}{\mathcal{O}}

\newcommand{\calS}{\mathcal{S}}

\def\Aut{{\text{Aut}}}

\def\Ker{{\text{Ker}}}
\def\Im{{\text{Im}}}
\def\O{{\text{O}}}

\def\Num{{\rm{Num}}}

\def\deg{{\text{deg}}}

\def\Num{{\text{Num}}}
\def\Km{{\text{Km}}}

\def\I{{\text{I}}}
\def\II{{\text{II}}}
\def\III{{\text{III}}}
\def\IV{{\text{IV}}}
\def\V{{\text{V}}}
\def\VI{{\text{VI}}}
\def\VII{{\text{VII}}}
\begin{document}
\title [Enriques surfaces] {On Enriques surfaces in characteristic 2 with a finite group of automorphisms}

\author{Toshiyuki Katsura}
\address{Faculty of Science and Engineering, Hosei University,
Koganei-shi, Tokyo 184-8584, Japan}
\email{toshiyuki.katsura.tk@hosei.ac.jp}

\author{Shigeyuki Kond\=o}
\address{Graduate School of Mathematics, Nagoya University, Nagoya,
464-8602, Japan}
\email{kondo@math.nagoya-u.ac.jp}
\thanks{Research of the first author is partially supported by
Grant-in-Aid for Scientific Research (B) No. 15H03614, and the second author by (S) No. 15H05738.}

\begin{abstract}
Complex Enriques surfaces with a finite group of automorphisms are
classified into seven types.  In this paper, we determine which types of such Enriques surfaces exist in characteristic 2.  In particular we give a one dimensional family of
classical and supersingular Enriques surfaces with the automorphism group $\Aut(X)$ isomorphic to the symmetric group $\mathfrak{S}_5$ of degree five. 
\end{abstract}

\maketitle

\section{Introduction}\label{sec1}

We work over an algebraically closed field $k$ of characteristic 2.
Complex Enriques surfaces with a finite group of automorphisms are completely
classified into seven types.  The main purpose of this paper 
is to determine which types of such Enriques surfaces exist in characteristic 2.
Recall that, over the complex numbers, a generic Enriques surface has an infinite group of automorphisms (Barth and Peters \cite{BP}).  On the other hand, Fano \cite{F} gave an Enriques surface with a finite group of automorphisms.  
Later Dolgachev \cite{D1} gave another example of such Enriques surfaces.  Then Nikulin \cite{N} proposed a classification of such Enriques surfaces in terms of the periods.  Finally the second author \cite{Ko} classified all complex Enriques surfaces
with a finite group of automorphisms, geometrically.  There are seven types ${\I, \II,\ldots, \VII}$ of such Enriques surfaces.  The Enriques surfaces of type ${\I}$ or 
${\II}$ form an irreducible one dimensional family, and each of the remaining types
consists of a unique Enriques surface.  
The first two types contain exactly twelve nonsingular rational curves, on the other hand, the remaining five types contain exactly twenty nonsingular rational curves.
The Enriques surface of type ${\I}$ (resp. of type ${\VII}$) is the example given 
by Dolgachev (resp. by Fano).   We call the dual graphs of all nonsingular rational curves on the Enriques surface of type $K$ the dual graph of type $K$ ($K = {\I, \II,..., \VII}$).  
 
In positive characteristics, the classification problem of Enriques surfaces with a finite group of automorphisms is still open.  Especially the case of characteristic 2 is most interesting.  In the paper \cite{BM2}, Bombieri and Mumford classified
Enriques surfaces in characteristic 2 into three classes, namely, singular, classical and supersingular Enriques surfaces.  
As in the case of characteristic $0$, an Enriques surface
$X$ in characteristic 2 has a canonical double cover
$\pi : Y \to X$, which is a separable ${\bf Z}/2{\bf Z}$-cover,
a purely inseparable $\mu_2$- or $\alpha_2$-cover according to $X$ being singular, classical or supersingular. The surface $Y$ might have singularities, but it is $K3$-like in the sense that its dualizing sheaf is trivial. 

In this paper we consider the following problem:
{\it does there exist an Enriques surface in characteristic $2$ with a finite group of automorphisms whose dual graph of all nonsingular rational curves is of type ${\rm I, II,..., VI}$ or ${\rm VII}$ $?$}  Note that if Enriques surface $S$ in any characteristic has the dual graph of
type $K$ ($K={\rm I, II,..., VII}$), then the automorphism group ${\rm Aut}(S)$ is finite by Vinberg's criterion (see Proposition \ref{Vinberg}).

We will prove the following Table \ref{Table1}:

\begin{table}[!htb]\label{}
{\offinterlineskip
\halign{\strut\vrule#&\quad\hfil\rm#\hfil\quad&&
\vrule#&\quad#\hfil\quad\cr\noalign{\hrule}
& {\rm Type}&&$\I$ && $\II$ && $\III$&&$\IV$&&$\V$ && $\VI$ &&$\VII$&\cr
\noalign{\hrule}
& {\rm singular} && {$\bigcirc$} && {$\bigcirc$} && {$\times$} && {$\times$} && {$\times$} && {$\bigcirc$} && {$\times$} &\cr
\noalign{\hrule}
& {\rm classical} && {$\times$} && {$\times$} && {$\times$} && {$\times$} && {$\times$} && {$\times$} && {$\bigcirc$} &\cr
\noalign{\hrule}
& {\rm supersingular} && {$\times$} && {$\times$} && {$\times$} && {$\times$} && {$\times$} && {$\times$} && {$\bigcirc$} &\cr
\noalign{\hrule}
}}
\
\caption{}
\label{Table1}
\end{table}

\noindent
In Table \ref{Table1}, $\bigcirc$ means the existence and $\times$ means the non-existence
of an Enriques surface with the dual graph of type $\I,..., \VII$.  

In case of type ${\I, \II, \VI}$, the construction of such Enriques surfaces over the complex numbers works well
in characteristic 2 (Theorems \ref{Ithm}, \ref{IIthm}, \ref{VIthm}).  The most difficult and interesting case is of type ${\VII}$.
We give a 1-dimensional family of classical 
and supersingular Enriques surfaces with a finite group of automorphisms whose dual graph is of type ${\VII}$ (Theorems \ref{main}, \ref{main2}).  We remark that this family is non-isotrivial (Theorem \ref{non-isotrivial}).
Recently the authors \cite{KK} gave a one dimensional family
of classical and supersingular Enriques surfaces which contain a remarkable forty divisors, by using a theory
of Rudakov and Shafarevich \cite{RS} on purely inseparable covers of surfaces.  We employ here the same method
to construct the above classical and supersingular Enriques surfaces with the dual graph of type ${\VII}$.

It is known that
there exist Enriques surfaces in characteristic 2 with a finite group of automorphisms
whose dual graphs of all nonsingular rational curves do not appear 
in the case of complex surfaces
(Ekedahl and Shepherd-Barron\cite{ES}, Salomonsson\cite{Sa}).  See Remark \ref{extra}.
The remaining problem of the classification of Enriques surfaces in characteristic 2 with 
a finite group of automorphisms is to determine such Enriques surfaces appeared only 
in characteristic 2.

The plan of this paper is as follows.  In section \ref{sec2}, we recall the known results
on Rudakov-Shafarevich's theory on derivations, lattices and Enriques surfaces.
In section \ref{sec3}, we give a construction of a one dimensional family of classical
and supersingular Enriques surfaces with the dual graph of type $\VII$.
Moreover we show the non-existence of singular Enriques surfaces with the dual graph of type ${\rm VII}$ (Theorem \ref{non-existVII}).
In section \ref{sec4}, we discuss other cases, that is, the existence of singular Enriques surfaces of type $\I, \II, \VI$ and the non-existence of other cases
(Theorems \ref{Ithm}, \ref{non-existI}, \ref{IIthm}, \ref{non-existII}, \ref{VIthm},
\ref{non-existVI}, \ref{non-existIII}).  
In appendices A and B, we give two remarks.  As appendix A, we show that 
the covering $K3$ surface of any singular Enriques surface has height $1$.  
As appendix B, we show that
for each singular Enriques surface with the dual graph of type ${\rm I}$ its canonical cover is isomorphic to the Kummer surface of the product of two ordinary elliptic curves. 

\medskip
\noindent
{\bf Acknowledgement.} The authors thank Igor Dolgachev for valuable conversations.
In particular all results in Section \ref{sec4} were obtained by discussion with him
in Soeul and Kyoto, 2014.
They thank him that he permits us to give these results in this paper.
The authors also thank Matthias Sch\"utt and Hiroyuki Ito for pointing out the non-existence of
singular Enriques surfaces with the dual graph of nonsingular rational curves of type
$\VII$.

\section{Preliminaries}\label{sec2}

Let $k$ be an algebraically closed field of characteristic $p > 0$,
and let $S$ be a nonsingular complete algebraic surface defined over $k$.
We denote by $K_{S}$ a canonical divisor of $S$.
A rational vector field $D$ on $S$ is said to be $p$-closed if there exists
a rational function $f$ on $S$ such that $D^p = fD$. 
A vector field $D$ for which $D^p=0$ is called of additive type, 
while that for which $D^p=D$ is called of multiplicative type.
Let $\{U_{i} = {\rm Spec} A_{i}\}$ be an affine open covering of $S$. We set 
$A_{i}^{D} = \{D(\alpha) = 0 \mid \alpha \in A_{i}\}$. 
Affine varieties $\{U_{i}^{D} = {\rm Spec} A_{i}^{D}\}$ glue together to 
define a normal quotient surface $S^{D}$.

Now, we  assume that $D$ is $p$-closed. Then,
the natural morphism $\pi : S \longrightarrow S^D$ is a purely
inseparable morphism of degree $p$. 
If the affine open covering $\{U_{i}\}$ of $S$ is fine enough, then
taking local coordinates $x_{i}, y_{i}$
on $U_{i}$, we see that there exsit $g_{i}, h_{i}\in A_{i}$ and 
a rational function $f_{i}$
such that the divisors defined by $g_{i} = 0$ and by $h_{i} = 0$ have no common divisor,
and such that
$$
 D = f_{i}\left(g_{i}\frac{\partial}{\partial x_{i}} + h_{i}\frac{\partial}{\partial y_{i}}\right)
\quad \mbox{on}~U_{i}.
$$
By Rudakov and Shafarevich \cite{RS} (Section 1), divisors $(f_{i})$ on $U_{i}$
give a global divisor $(D)$ on $S$, and zero-cycles defined
by the ideal $(g_{i}, h_{i})$ on $U_{i}$ give a global zero cycle 
$\langle D \rangle $ on $S$. A point contained in the support of
$\langle D \rangle $ is called an isolated singular point of $D$.
If $D$ has no isolated singular point, $D$ is said to be divisorial.
Rudakov and Shafarevich (\cite{RS}, Theorem 1, Corollary) 
showed that $S^D$ is nonsingular
if $\langle D \rangle  = 0$, i.e., $D$ is divisorial.
When $S^D$ is nonsingular,
they also showed a canonical divisor formula
\begin{equation}\label{canonical}
K_{S} \sim \pi^{*}K_{S^D} + (p - 1)(D),
\end{equation}
where $\sim$ means linear equivalence.
As for the Euler number $c_{2}(S)$ of $S$, we have a formula
\begin{equation}\label{euler}
c_{2}(S) = \deg \langle D \rangle  - \langle K_{S}, (D)\rangle - (D)^2
\end{equation}
(cf. Katsura and Takeda \cite{KT}, Proposition 2.1). 

Now we consider an irreducible curve $C$ on $S$ and we set $C' = \pi (C)$.
Take an affine open set $U_{i}$ above such that $C \cap U_{i}$ is non-empty.
The curve $C$ is said to be integral with respect to the vector field $D$
if $g_{i}\frac{\partial}{\partial x_{i}} + h_{i}\frac{\partial}{\partial y_{i}}$
is tangent to $C$ at a general point of $C \cap U_{i}$. Then, Rudakov-Shafarevich
\cite{RS} (Proposition 1) showed the following proposition:

\begin{prop}\label{insep}

$({\rm i})$  If $C$ is integral, then $C = \pi^{-1}(C')$ and $C^2 = pC'^2$.

$({\rm ii})$  If $C$ is not integral, then $pC = \pi^{-1}(C')$ and $pC^2 = C'^2$.
\end{prop}

A lattice is a free abelian group $L$ of finite rank equipped with 
a non-degenerate symmetric integral bilinear form $\langle . , . \rangle : L \times L \to {\bf Z}$. 
The signature of a lattice is the signature of the real vector space $L\otimes {\bf R}$ equipped with the symmetric bilinear form extended from the one on $L$ by linearity. A lattice is called even if 
$\langle x, x\rangle \in 2{\bf Z}$ 
for all $x\in L$. 
We denote by $U$ the even unimodular lattice of signature $(1,1)$, 
and by $A_m, \ D_n$ or $\ E_k$ the even {\it negative} definite lattice defined by
the Cartan matrix of type $A_m, \ D_n$ or $\ E_k$ respectively.    
We denote by $L\oplus M$ the orthogonal direct sum of lattices $L$ and $M$.
Let ${\rm O}(L)$ be the orthogonal group of $L$, that is, the group of isomorphisms of $L$ preserving the bilinear form.

In characteristic 2, a minimal algebaic surface with numerically trivial
canonical divisor is called an Enriques surface if the second Betti
number is equal to 10. Such surfaces $S$ are  divided into three classes
(for details, see Bombieri and Mumford \cite{BM2}, Section 3):
\begin{itemize}
\item[$({\rm i})$] $K_{S}$ is not linearly equivalent to zero 
and $2K_{S}\sim 0$.  Such an Enriques surface is called a classical Enriques surface.
\item[$({\rm ii})$] $K_{S} \sim 0$, ${\rm H}^{1}(S, {\calO}_{S}) \cong k$
and the Frobenius map acts on  ${\rm H}^{1}(S, {\calO}_S)$ bijectively.
Such an Enriques surface is called a singular Enriques surface.
\item[$({\rm iii})$] $K_{S} \sim 0$, ${\rm H}^{1}(S, {\calO}_{S}) \cong k$
and the Frobenius map is the zero map on  ${\rm H}^{1}(S, {\calO}_S)$.
Such an Enriques surface is called a supersingular Enriques surface.
\end{itemize}

Let $S$ be an Enriques surface and let $\Num(S)$ be the quotient of the N\'eron-Severi group of $S$ by torsion.  Then $\Num(S)$ together with the intersection product is
an even unimodular lattice of signature $(1,9)$ (Cossec and Dolgachev \cite{CD}, Chap. II, Theorem 2.5.1), and hence is isomorphic to $U\oplus E_8$. 
We denote by ${\rm O}(\Num(S))$ the orthogonal group of $\Num(S)$. The set 
$$\{ x \in \Num(S)\otimes {\bf R} \ : \ \langle x, x \rangle > 0\}$$ 
has two connected components.
Denote by $P(S)$ the connected component containing an ample class of $S$.  
For $\delta \in \Num(S)$ with $\delta^2=-2$, we define
an isometry $s_{\delta}$ of $\Num(S)$ by
$$s_{\delta}(x) = x + \langle x, \delta\rangle \delta, \quad x \in \Num(S).$$ 
The isometry $s_{\delta}$ is called the reflection associated with $\delta$.
Let $W(S)$ be the subgroup of
${\rm O}(\Num(S))$ generated by reflections associated with all nonsingular rational  curves on $S$.  Then $P(S)$ is divided into chambers 
each of which is a fundamental domain with respect to
the action of $W(S)$ on $P(S)$.
There exists a unique chamber containing an ample
class which is nothing but the closure of the ample cone $D(S)$ of $S$.
It is known that the natural map
\begin{equation}\label{coh-trivial}
\rho : \Aut(S) \to {\rm O}(\Num(S))
\end{equation}
has a finite kernel (Dolgachev \cite{D2}, Theorems 4, 6).  
Since the image $\Im(\rho)$ preserves the ample cone, we see $\Im(\rho) \cap W(S) = \{1\}$.
Therefore $\Aut(S)$ is finite if the index $[\O(\Num(S)) : W(S)]$ is finite.  
Thus we have the following Proposition (see Dolgachev \cite{D1}, Proposition 3.2).

\begin{prop}\label{finiteness}
If $W(S)$ is of finite index in ${\rm O}({\rm Num}(S))$, then ${\rm Aut}(S)$ is finite.
\end{prop}

\noindent
Over the field of complex numbers, the converse of Proposition \ref{finiteness} 
holds by using the Torelli type theorem for Enriques surfaces (Dolgachev \cite{D1}, Theorem 3.3).

Now, we recall Vinberg's criterion
which guarantees that a group generated by finite number of reflections is
of finite index in ${\rm O}(\Num(S))$.

Let $\Delta$ be a finite set of $(-2)$-vectors in $\Num(S)$.
Let $\Gamma$ be the graph of $\Delta$, that is,
$\Delta$ is the set of vertices of $\Gamma$ and two vertices $\delta$ and $\delta'$ are joined by $m$-tuple lines if $\langle \delta, \delta'\rangle=m$.
We assume that the cone
$$K(\Gamma) = \{ x \in \Num(S)\otimes {\bf R} \ : \ \langle x, \delta_i \rangle \geq 0, \ \delta_i \in \Delta\}$$
is a strictly convex cone. Such $\Gamma$ is called non-degenerate.
A connected parabolic subdiagram $\Gamma'$ in $\Gamma$ is a  Dynkin diagram of type $\tilde{A}_m$, $\tilde{D}_n$ or $\tilde{E}_k$ (see \cite{V}, p. 345, Table 2).  If the number of vertices of $\Gamma'$ is $r+1$, then $r$ is called the rank of $\Gamma'$.  A disjoint union of connected parabolic subdiagrams is called a parabolic subdiagram of $\Gamma$.  We denote by $\tilde{K_1}\oplus \tilde{K_2}$ a parabolic subdiagram which is a disjoint union of two
connected parabolic subdiagrams of type $\tilde{K_1}$ and $\tilde{K_2}$, where
$K_i$ is $A_m$, $D_n$ or $E_k$. The rank of a parabolic subdiagram is the sum of the rank of its connected components.  Note that the dual graph of singular fibers of an elliptic fibration on $S$ gives a parabolic subdiagram.  For example, a singular fiber of type ${\rm III}$, ${\rm IV}$ or ${\rm I}_{n+1}$ defines a parabolic subdiagram of type $\tilde{A}_1$, $\tilde{A}_2$ or 
$\tilde{A}_n$ respectively.  
We denote by $W(\Gamma)$ the subgroup of ${\rm O}(\Num(S))$ 
generated by reflections associated with $\delta \in \Gamma$.

\begin{prop}\label{Vinberg}{\rm (Vinberg \cite{V}, Theorem 2.3)}
Let $\Delta$ be a set of $(-2)$-vectors in $\Num(S)$
and let $\Gamma$ be the graph of $\Delta$.
Assume that $\Delta$ is a finite set, $\Gamma$ is non-degenerate and $\Gamma$ contains no $m$-tuple lines with $m \geq 3$.  Then $W(\Gamma)$ is of finite index in ${\rm O}(\Num(S))$ if and only if every connected parabolic subdiagram of $\Gamma$ is a connected component of some
parabolic subdiagram in $\Gamma$ of rank $8$ {\rm (}= the maximal one{\rm )}.
\end{prop} 
\noindent

Finally we recall some facts on elliptic fibrations on Enriques surfaces.

\begin{prop}\label{multi-fiber}{\rm (Dolgachev and Liedtke \cite{DL}, Theorem 4.8.3)}

Let $f : S \to {\bf P}^1$ be an elliptic fibration on an Enriques surface $S$ in 
characteristic $2$.  Then the following hold.

$({\rm i})$  If $S$ is classical, then $f$ has two tame multiple fibers, each is
either an ordinary elliptic curve or a singular fiber of additive type.

$({\rm ii})$  If $S$ is singular, then $f$ has one wild multiple 
fiber which is a smooth ordinary elliptic curve or a singular fiber of multiplicative type.

$({\rm iii})$  If $S$ is supersingular, then $f$ has one wild multiple fiber which is a
supersingular elliptic curve or a singular fiber of additive type.
\end{prop}
\begin{proof}
As for the number of multiple fibers in each case, it is given 
in Bombieri and Mumford \cite{BM2}, Proposition 11.
Let $2G$ be a multiple fiber of $f : S \longrightarrow {\bf P}^1$.
If $S$ is classical, the multiple fiber $2G$ is tame.
Therefore, the normal bundle ${\calO}_{G}(G)$ of $G$ is of order 2 (cf. Katsura and Ueno \cite{KU},
p. 295, (1.7)). On the other hand, neither the Picard variety ${\rm Pic}^0({\bf G}_m)$
of the multiplicative group ${\bf G}_m$ nor ${\rm Pic}^0(E)$ of the supersingular elliptic curve
$E$ has any 2-torsion point. Therefore, $G$ is either an ordinary elliptic curve or
a singular fiber of additive type. Now, we consider an exact sequence:
$$
  0 \longrightarrow {\calO}_S(-G)  \longrightarrow {\calO}_S \longrightarrow {\calO}_{G}
\longrightarrow 0.
$$
Then, we have the long exact sequence
$$
\rightarrow H^1(S, {\calO}_S) \longrightarrow H^1(G, {\calO}_G)
\longrightarrow H^2(S, {\calO}_S(-G)) \longrightarrow H^2(S, {\calO}_S)\rightarrow 0.
$$
If $S$ is either singular or supersingular, we have 
$H^1(S, {\calO}_S)\cong H^2(S, {\calO}_S)\cong k$. 
Note that in our case the canonical divisor $K_S$ is linearly equivalent to 0.
Since $2G$ is a multiple fiber,
by the Serre duality theorem, we have
$$
H^2(S, {\calO}_S(-G)) \cong H^0(S, {\calO}_S(K_S + G)) \cong H^0(S, {\calO}_S(G))\cong k.
$$ 
Therefore, we see that
the natural homomorphism 
$$
H^1(S, {\calO}_S) \longrightarrow H^1(G, {\calO}_G)
$$
is an isomorphism. If $S$ is singular, then the Frobenius map $F$ acts bijectively
on $H^1(S, {\calO}_S)$. Hence, $F$ acts on $H^1(G, {\calO}_G)$ bijectively.
Therefore, $G$ is either an ordinary elliptic curve or a singular fiber of multiplicative type.
If $S$ is supersingular, then the Frobenius map $F$ is the zero map
on $H^1(S, {\calO}_S)$. Hence, $F$ is also a zero map on $H^1(G, {\calO}_G)$.
Therefore, $G$ is either a supersingular elliptic curve or a singular fiber of additive type.
\end{proof}

Let $f : S \to {\bf P}^1$ be an elliptic fibration on an Enriques surface $S$.
We use Kodaira's notation for singular fibers of $f$:
$${\rm I}_n,\ {\rm I}_n^*,\ {\rm II},\ {\rm II}^*,\ {\rm III},\ {\rm III}^*,\ {\rm IV},\  {\rm IV}^*.$$

\begin{prop}\label{singular-fiber}

Let $f : S \to {\bf P}^1$ be an elliptic fibration on an Enriques surface $S$ in 
characteristic $2$.  Then the type of reducible singular fibers is one of the following:
$$({\rm I}_3, {\rm I}_3, {\rm I}_3, {\rm I}_3), \ ({\rm I}_5, {\rm I}_5), \
({\rm I}_9),\ ({\rm I}_4^*),\ ({\rm II}^*),\ ({\rm III}, {\rm I}_8),$$
$$({\rm I}_1^*, {\rm I}_4), \ ({\rm III}^*, {\rm I}_2),\
({\rm IV}, {\rm IV}^*),\ ({\rm IV}, {\rm I}_2, {\rm I}_6),\ ({\rm IV}^*, {\rm I}_3).$$
\end{prop}
\begin{proof}
Consider the Jacobian fibration $J(f) : R \to {\bf P}^1$ of $f$ which is a rational elliptic surface.  
It is known that the type of singular fibers of $f$ coincides with that of
$J(f)$ (cf. Liu-Lorenzini-Raynaud \cite{LLR}, Theorem 6.6).  Now the assertion follows from the classification of singular fibers 
of rational elliptic surfaces in characteristic 2 due to Lang \cite{L1}, \cite{L2} (also see Ito \cite{I}).
\end{proof}

\section{Enriques surfaces with the dual graph of type VII}\label{sec3}

In this section, we construct Enriques surfaces in characteristic 2 whose dual graph 
of all nonsingular rational curves is of type VII.
The method to construct them is similar to the one in Katsura and Kondo \cite{KK}, \S 4.

We consider the nonsingular complete model
of the supersingular elliptic curve $E$ defined by
$$
         y^2 + y = x^3 + x^2.
$$
For $(x_1, y_1), (x_2, y_2) \in E$, the addition of this elliptic curve is given by, 
$$
\begin{array}{l}
x_{3} = x_1 + x_2 +
\left(\frac{y_2 + y_1}{x_2 + x_1}\right)^2 + 1 \\
y_3 = y_1 + y_2 + \left(\frac{y_2 + y_1}{x_2 + x_1}\right)^3 + \left(\frac{y_2 + y_1}{x_2 + x_1}\right) + \frac{x_1y_2 +x_2y_1}{x_2 +x_1} + 1.
\end{array}
$$
The ${\bf F}_4$-rational points of $E$ are given by
$$
\begin{array}{l}
P_{0} = \infty,  P_{1} =(1, 0), P_{2} =(0, 0), P_{3} =(0, 1),
 P_{4} =(1, 1).
\end{array}
$$
The point $P_{0}$ is the zero point of $E$, and these points make
the cyclic group of order five : 
$$
   P_{i} = iP_{1} \quad (i = 2, 3, 4),~P_{0} = 5P_{1}
$$
Now we consider the relatively minimal
nonsingular complete  elliptic surface $\psi : R \longrightarrow {\bf P}^1$
defined by
$$
y^2 + sxy + y = x^3 + x^2 + s
$$
with a parameter $s$.
This surface is a rational elliptic surface with two singular fibers of type $\I_5$
over the points given by $s = 1, \infty$, and two singular fibers of type $\I_1$
over the points given by $t = \omega, \omega^2$.
Here, $\omega$ is a primitive cube root of unity.
We consider the base change of $\psi : R \longrightarrow {\bf P}^1$
by $s = t^2$. 
Then, we have the elliptic surface defined by 
$$
(*)\quad \quad \quad   y^2 + t^2xy + y = x^3 + x^2 + t^2.
$$
We consider the relatively minimal nonsingular complete model 
of this elliptic surface :
\begin{equation}\label{pencil3}
f : Y \longrightarrow {\bf P}^1.
\end{equation}
The surface $Y$ is an elliptic $K3$ surface. 
From $Y$ to $R$, there exists a generically surjective 
purely inseparable rational map. We denote by $R^{(\frac{1}{2})}$
the algebraic surface whose coefficients of the defining equations are the square
roots of those of $R$. Then, $R^{(\frac{1}{2})}$ is also a rational surface, and
we have the Frobenius morphism $F : R^{(\frac{1}{2})}\longrightarrow R$. $F$ factors
through a generically surjective 
purely inseparable rational map from $R^{(\frac{1}{2})}$ to $Y$. 
By the fact that $R^{(\frac{1}{2})}$ is rational
we see that $Y$ is unirational. Hence, $Y$
is a supersingular $K3$ surface, i.e. the Picard number $\rho (Y)$ is equal 
to the second Betti number $b_{2}(Y)$ (cf. Shioda \cite{S}, p.235, Corollary 1).

The discriminant of the elliptic surface $f : Y \longrightarrow {\bf P}^1$ is given by
$$
   \Delta = (t + 1)^{10}(t^2 + t + 1)^2
$$
and the $j$-invariant is given by
$$
   j = t^{24}/(t + 1)^{10}(t^2 + t + 1)^2.
$$
Therefore,
on the elliptic surface $f : Y \longrightarrow {\bf P}^1$,
there exist two singular fibers of type $\I_{10}$ over
the points given by $t = 1, \infty$, and two singular fibers 
of type $\I_2$ over
the points given by $t = \omega, \omega^2$.
The regular fiber over the point defined by $t = 0$ is
the supersingular elliptic curve $E$.  

The elliptic $K3$ surface $f: Y \longrightarrow {\bf P}^1$
has ten sections $s_{i}, m_{i}$ $(i = 0, 1, 2, 3, 4)$ given as follows:
$$
\begin{array}{ll}
s_0 : \mbox{the zero section} &\mbox{passing through}~P_{0}~\mbox{on}~E\\
s_1 : x = 1, y = t^2 &\mbox{passing through}~P_{1}~\mbox{on}~E\\
s_2 : x = t^2, y = t^2 &\mbox{passing through}~P_{2}~\mbox{on}~E\\
s_3 : x = t^2, y = t^4 + t^2 + 1&\mbox{passing through}~P_{3}~\mbox{on}~E\\
s_4 :  x = 1, y = 1 &\mbox{passing through}~P_{4}~\mbox{on}~E\\
m_0 : x = \frac{1}{t^2}, y = \frac{1}{t^3} +\frac{1}{t^2} + t &\mbox{passing through}~P_{0}~\mbox{on}~E\\
m_1 : x = t^3 + t + 1,~ y = t^4 + t^3 + t &\mbox{passing through}~P_{1}~\mbox{on}~E\\
m_2 : x = t,~ y = t^3 &\mbox{passing through}~P_{2}~\mbox{on}~E\\
m_3 :  x = t,~ y = 1&\mbox{passing through}~P_{3}~\mbox{on}~E\\
m_4 :  x = t^3 + t + 1,~ y = t^5 + t^4 + t^2 + t + 1&\mbox{passing through}~P_{4}~\mbox{on}~E.
\end{array}
$$
These ten sections make the cyclic group of order 10, and the group structure is
given by
$$
s_{i} = is_1,\ m_i =m_0 + s_i~(i = 0, 1, 2, 3, 4),\ 2m_0 = s_0
$$
with $s_0$, the zero section. 
The images of $s_{i}$ (resp. $m_{i}$) ($i = 0, 1, 2, 3, 4$) on $R$ give sections 
(resp. multi-sections) of $\psi : R \longrightarrow {\bf P}^1$.
The intersection numbers of the sections $s_i, m_i$ $(i = 0, 1, 2, 3, 4)$
are given by
\begin{equation}\label{int-sections}
   \langle s_i, s_j\rangle =-2\delta_{ij},~ \langle m_i, m_j\rangle~=-2\delta_{ij},~ \langle s_i, m_j\rangle = \delta_{ij},~ 
\end{equation}
where $\delta_{ij}$ is Kronecker's delta.

On the singular elliptic surface $(*)$, we denote by $F_1$ the fiber
over the point defined by $t = 1$. $F_1$ is an irreducible curve and on $F_1$
the surface $(*)$ has only one singular point $P$.
The surface $Y$ is a surface obtained by the minimal resolution of singularities of $(*)$.
We denote the proper transform of $F_1$ on Y again by $F_1$, if confusion doesn't occur.
We have nine exceptional curves $E_{1,i}$ $(i = 1,2, \ldots, 9)$ over the point $P$, and
as a singular fiber of type $I_{10}$ of the elliptic surface $f : Y \longrightarrow {\bf P}^1$,
$F_1$ and these nine exceptional curves make a decagon $F_1E_{1,1}E_{1,2}\ldots E_{1,9}$ clockwisely.
The blowing-up at the singular point $P$ gives two exceptional curves
$E_{1,1}$ and $E_{1,9}$, and they intersect each other at a singular point. The blowing-up
at the singular point again gives two exceptional curves $E_{1,2}$ and $E_{1,8}$. 
The exceptional curve  $E_{1,2}$ (resp. $E_{1,8}$) intersects $E_{1,1}$ (resp. $E_{1,9}$) transeversely.
Exceptional curves $E_{1,2}$ and $E_{1,8}$ intersect each other at a singular point, and so on. 
By successive blowing-ups, 
the exceptional curve $E_{1,5}$ finally appears to complete the resolution of singularity 
at the point $P$, and it intersects $E_{1,4}$ and $E_{1,6}$ transeversely. 
Summerizing these results, we see that $F_1$ intersects $E_{1,1}$ and $E_{1,9}$ transversely, and that $E_{1,i}$ intersects $E_{1,i+ 1}$ $(i = 1, 2, \ldots, 8)$ transversely. 
We choose $E_{1,1}$ as the component
which intersects the section $m_2$. Then,
10 sections above intersect these 10 curves transversely as follows:

\begin{center}
\begin{tabular}{|c|c|c|c|c|c|c|c|c|c|c|}
\hline
sections & $s_0$ &$s_1$ &$s_2$ &$s_3$ &$s_4$ &$m_0$ &$m_1$ &$m_2$ & $m_3$ &$m_4$ \\
\hline
componets  &$F_1$ &$E_{1,8}$ &$E_{1,6}$ &$E_{1,4}$ &$E_{1,2}$ &$E_{1,5}$ &$E_{1,3}$ & $E_{1,1}$ & $E_{1,9}$ & $E_{1,7}$ \\
\hline
\end{tabular}
\end{center}
\noindent
Here, the table means that the section $s_0$ intersects the singular fiber 
over the point defined by 
$t= 1$ with the component $F_1$, for example.

The surface $Y$ has the automorphism $\sigma$ defined by
$$
 (t, x, y) \mapsto (\frac{t}{t+1}, \frac{x + t^4 + t^2 + 1}{(t + 1)^4}, 
\frac{x + y +s^6 + s^2}{(s + 1)^6}).
$$
The automorphism $\sigma$ is of order 4 and replaces the fiber over the point $t = 1$  
with the one
over the point $t = \infty$, and also replaces the fiber over the point $t =\omega$
with the one over the point $t = \omega^2$. The automorphism $\sigma$ acts 
on the ten sections above as follows:
\begin{center}
\begin{tabular}{|c|c|c|c|c|c|c|c|c|c|c|}
\hline
sections & $s_0$ &$s_1$ &$s_2$ &$s_3$ &$s_4$ &$m_0$ &$m_1$ &$m_2$ & $m_3$ &$m_4$ \\
\hline
$\sigma^{*}$(sections) &$s_0$ &$s_2$ &$s_4$ &$s_1$ &$s_3$ &$m_0$ &$m_2$ & $m_4$ & $m_1$ & $m_3$ \\
\hline
\end{tabular}
\end{center}
Using the automorphism $\sigma$, to construct the resolution of singularity on the fiber
over the point $P_{\infty}$ defined by $t = \infty$, we use the resolution of singularity 
on the fiber over the point $P_{1}$ defined by $t = 1$.
We attach names to the irreducible components of the fiber over $P_{\infty}$
in the same way as above.
Namely,
on the singular elliptic surface $(*)$, we denote by $F_{\infty}$ the fiber
over the point defined by $t = \infty$. We also denote the proper transform 
of $F_{\infty}$  on $Y$  by $F_{\infty}$. 
We have 9 exceptinal curves $E_{\infty,i}$ $(i = 1,2, \ldots, 9)$ over the point $P_{\infty}$, and
as a singular fiber of type $I_{10}$ of the elliptic surface $f : Y \longrightarrow {\bf P}^1$,
$F_{\infty}$ and these 9 exceptional curves make a decagon $F_{\infty}E_{\infty, 1}E_{\infty, 2}\ldots E_{\infty, 9}$ clockwisely. $F_{\infty}$ intersects $E_{\infty, 1}$ and $E_{\infty, 9}$ transversely, and that $E_{\infty, i}$ intersects $E_{\infty, i+ 1}$ 
$(i = 1, 2, \ldots, 8)$ transversely. 

The singular fiber of $f : Y \longrightarrow {\bf P}^1$ over the point defined by $t= \omega$
(resp. $t = \omega^{2}$) consists of two irreducible components $F_{\omega}$
and $E_{\omega}$ (resp. $F_{\omega^{2}}$ and $E_{\omega^{2}}$), 
where $F_{\omega}$ (resp. $F_{\omega^{2}}$) 
is the proper transform of the fiber over the point $P_{\omega}$ 
(resp. $P_{\omega^{2}}$) in $(*)$. 

Then,
the 10 sections above intersect singular fibers of elliptic surface $f : Y \longrightarrow {\bf P}^1$ as follows:
{
\footnotesize
\begin{center}
\begin{table}[!htb]\label{}
\begin{tabular}{|l|c|c|c|c|c|c|c|c|c|c|}
\hline
sections & $s_0$ &$s_1$ &$s_2$ &$s_3$ &$s_4$ &$m_0$ &$m_1$ &$m_2$ & $m_3$ &$m_4$ \\
\hline
 $t= 1$  &$F_1$ &$E_{1,8}$ &$E_{1,6}$ &$E_{1,4}$ &$E_{1,2}$ &$E_{1,5}$ &$E_{1,3}$ & $E_{1,1}$ & $E_{1,9}$ & $E_{1,7}$ \\
\hline
 $t = \infty$  &$F_\infty$ &$E_{\infty, 6}$ &$E_{\infty, 2}$ &$E_{\infty, 8}$ &$E_{\infty, 4}$ &$E_{\infty, 5}$ &$E_{\infty, 1}$ & $E_{\infty, 7}$ & $E_{\infty, 3}$ & $E_{\infty, 9}$ \\
\hline
 $t = \omega$ & $F_\omega$ &$F_\omega$ &$F_\omega$ &$F_\omega$ &$F_\omega$ &$E_{\omega}$ &$E_{\omega}$ & $E_{\omega}$ & $E_{\omega}$ & $E_{\omega}$ \\
\hline
 $t = \omega^2$  &$F_{\omega^2}$ &$F_{\omega^2}$ &$F_{\omega^2}$ &$F_{\omega^2}$ &$F_{\omega^2}$ &$E_{\omega^2}$ &$E_{\omega^2}$ & $E_{\omega^2}$ & $E_{\omega^2}$ & $E_{\omega^2}$ \\
\hline
\end{tabular}
\caption{}
\label{Table2}
\end{table}
\end{center}  
}
\begin{prop}\label{}
The surface $Y$ is a supersingular $K3$ surface with Artin invariant $1$.
\end{prop}
\begin{proof}
The elliptic fibration $(\ref{pencil3})$ has two singular fibers of type 
$\I_{10}$, two singular fibers of type $\I_2$ and
ten sections.  Hence the assertion follows 
from the Shioda-Tate formula (cf. Shioda \cite{Shio}, Corollary 1.7).
\end{proof}

Incidentally, by the Shioda-Tate formula, we also see that the order of 
the group of the sections of $f : Y \longrightarrow {\bf P}^1$ is equal to 10 
and so the group is isomorphic to ${\bf Z}/10{\bf Z}$.

Now, we consider a rational vector field
$$
 D' = (t - 1)(t - a)(t - b)\frac{\partial}{\partial t} + (1 + t^2x)\frac{\partial}{\partial x}
$$
with $a, b \in k, \ a+b=ab, \ a^3\not=1$.
Then, we have $D'^2 = t^2 D'$, that is, $D'$ is $2$-closed. 
On the surface $Y$, the divisorial part of $D'$ is given by
$$
\begin{array}{rl}
(D') & = E_{1,1} + E_{1,3} + E_{1,5} + E_{1,7} + E_{1,9} + E_{\infty, 1} + E_{\infty, 3}
+ E_{\infty, 5} + E_{\infty, 7} \\
 & + E_{\infty, 9} -  E_{\omega} - E_{\omega^{2}} - 2(F_{\infty} + E_{\infty, 1} + E_{\infty, 2}+ E_{\infty, 3} 
+ E_{\infty, 4} + E_{\infty, 5} \\
 &+ E_{\infty, 6} + E_{\infty, 7} + E_{\infty, 8} + E_{\infty, 9}).
\end{array}
$$

We set $D = \frac{1}{t - 1}D'$. Then, $D^2=abD$, that is, $D$ is also 2-closed 
and $D$ is of additive type if $a=b=0$ and
of multiplicative type  otherwise.  Moreover, we have
\begin{equation}\label{divisorial}
\begin{array}{rl}
   (D) & =  - (F_{1} + E_{1,2} + E_{1,4} + E_{1,6} +  E_{1,8} 
+ F_{\infty} +  E_{\infty, 2} + E_{\infty, 4} \\
 &+  E_{\infty, 6} +  E_{\infty, 8} + E_{\omega} + E_{\omega^2}).
\end{array}
\end{equation}

From here  until Theorem \ref{main}, the argument is parallel to the one
in Katsura and Kondo \cite{KK}, \S 4, and so we give just a brief sketch of the proofs 
for the readers' convenience.
\begin{lemma}
The quotient surface $Y^{D}$ is nonsingular.
\end{lemma}
 \begin{proof}
Since $Y$ is a $K3$ surface, we have $c_{2}(Y) = 24$.
Using $(D)^2 = -24$ and the equation (\ref{euler}), we have 
$$
24 = c_{2}(Y) = \deg \langle D\rangle - \langle K_{Y}, (D)\rangle - (D)^2 = \deg \langle D\rangle  + 24.
$$
Therefore, we have $\deg  \langle D\rangle = 0$. This means that $D$ is divisorial, and that 
$Y^{D}$ is nonsingular.
\end{proof}

By the result on the canonical divisor formula of Rudakov and Shafarevich (see the equation (\ref{canonical})),
we have
$$
        K_{Y} = \pi^{*} K_{Y^D} + (D).
$$
\begin{lemma}\label{exceptional}
Let $C$ be an irreducible curve contained in the support of the divisor $(D)$,
and set $C' = \pi (C)$. Then, $C'$ is an exceptional curve of the first kind.
\end{lemma}
\begin{proof}
By direct calculation, $C$ is integral with respect to $D$. Therefore,
we have $C = \pi^{-1}(C')$ by Proposition \ref{insep}.
By the equation $2C'^2 = (\pi^{-1}(C'))^2 = C^2 = - 2$, we have $C'^2 = -1$.
Since $Y$ is a $K3$ surface, $K_Y$ is 
linearly equivalent to zero.
Therefore, we have
$$
2\langle K_{Y^D}, C'\rangle  =  \langle \pi^{*}K_{Y^D}, \pi^{*}(C')\rangle\\
    =  \langle K_Y - (D), C\rangle = C^2 = -2.
$$
Therefore, we have $\langle K_{Y^D}, C'\rangle  = -1$ and 
the arithmetic genus of $C'$ is equal to $0$.
Hence, $C'$ is an exceptional curve of the first kind.
\end{proof}

We denote these 12 exceptional curves on $Y^{D}$ by $E'_{i}$ ($i = 1, 2, \ldots, 12$),
which are the images of irreducible components of $-(D)$ by $\pi$.
Let 
$$\varphi : Y^{D} \to X_{a,b}$$ 
be the blowing-downs of $E'_{i}$ ($i = 1, 2, \ldots, 12$).  For simplicity, we denote $X_{a,b}$ by $X$.
Now we have the following commutative diagram:
$$
\begin{array}{ccc}\label{maps}
       \quad    Y^{D} & \stackrel{\pi}{\longleftarrow} & Y \\
                \varphi \downarrow &    &   \downarrow f \\
      \quad        X=X_{a,b}     &        &   {\bf P}^1 \\
           g \downarrow & \quad   \swarrow_{F}&  \\
    \quad        {\bf P}^1 &   &
\end{array}
$$
Here $F$ is the Frobenius base change.
Then, we have
$$
         K_{Y^D} = \varphi^{*}(K_{X}) + \sum_{i = 1}^{12}E'_{i}.
$$
\begin{lemma} 
The canonical divisor $K_{X}$ of $X$ is numerically equivalent to $0$.
\end{lemma}
\begin{proof}
As mentioned in the proof of Lemma \ref{exceptional}, all irreducible curves which appear
in the divisor $(D)$ are integral with respect to the vector field $D$.
For an irreducible component $C$ of $(D)$, we denote by $C'$ the image $\pi (C)$ of $C$.
Then, we have $C = \pi^{-1}(C')$ by Proposition \ref{insep}. Therefore, we have
$$
       (D) = - \pi^{*}(\sum_{i = 1}^{12}E'_{i}).
$$
Since $Y$ is a $K3$ surface,
$$
     0 \sim K_{Y}  = \pi^{*}K_{Y^D} + (D) 
 = \pi^{*}( \varphi^{*}(K_{X}) + \sum_{i = 1}^{12}E'_{i})  + (D) = \pi^{*}(\varphi^{*}(K_{X}))
$$
Therefore, $K_{X}$ is numerically equivalent to zero.
\end{proof}

\begin{lemma} 
The surface $X$ has $b_{2}(X) = 10$ and $c_{2}(X) = 12$.
\end{lemma}
\begin{proof}
Since $\pi : {Y} \longrightarrow {Y}^{D}$ is finite and
purely inseparable, the \'etale cohomology of $Y$ is isomorphic to 
the \'etale cohomology of $Y^{D}$. Therefore, we have
$b_{1}(Y^{D}) = b_{1}(Y) = 0$, 
$b_{3}(Y^{D})= b_{3}(Y) = 0$ and $b_{2}(Y^{D}) 
= b_{2}(Y) = 22$. Since $\varphi$ is the blowing-downs
of 12 exceptional curves of the first kind, we see
$b_{0}(X) =b_{4}(X) = 1$, $b_{1}(X) =b_{3}(X) = 0$ and $b_{2}(X) = 10$.
Therefore, we have 
$$
c_{2}(X) = b_{0}(X) - b_{1}(X) + b_{2}(X) -b_{3}(X) + b_{4}(X) = 12.
$$
\end{proof}

\begin{theorem}\label{main}
Under the notation above, the following statements hold.
\begin{itemize}
\item[$({\rm i})$] The surface $X=X_{a,b}$ is a supersingular Enriques surface 
if $a = b = 0$. 
\item[$({\rm ii})$] The surface $X=X_{a,b}$ is a classical Enriques surface 
if $a + b = ab$ and $a \notin {\bf F}_{4}$.
\end{itemize}
\end{theorem}
\begin{proof}
Since $K_X$ is numerically trivial,
$X$ is minimal and the Kodaira dimension $\kappa (X) $  is equal to $0$.
Since $b_2(X) = 10$, $X$ is an Enriques surface.
Since $Y$ is a supersingular K3 surface, $X$ is either supersingular or classical.

In case that $a= b = 0$,  the integral fiber of the elliptic fibration $f : Y \longrightarrow {\bf P}^1$ 
with respect to $D$ exists only over the point $P_{0}$ defined by $t = 0$.
Hence $g : X \longrightarrow {\bf P}^1$ has only one multiple fiber.
Therefore, the multiple fiber is wild, and $X$ is a supersingular Enriques surface.
In case that $a \not\in  {\bf F}_4$, the integral fibers of 
the elliptic fibration $f : Y \longrightarrow {\bf P}^1$ 
with respect to $D$ exist over the points $P_{a}$ defined by $t = a$
and $P_b$ defined by $t = b$. Therefore, the multiple fibers are tame, and 
we conclude that $X$ is a classical Enriques surface.
\end{proof}

Recall that the elliptic fibration $f : Y \to {\bf P}^1$ given in (\ref{pencil3})
has two singular fibers of type $\I_{10}$, two singular fibers of type $\I_2$ and
ten sections.  This fibration induces an elliptic fibration
$$g : X\to {\bf P}^1$$
which has two singular fibers of type $\I_5$, two singular fibers of type $\I_1$, and
ten 2-sections.  
Thus we have twenty nonsingular rational curves on $X$.
Denote by $\calE$ the set of curves contained in the support of the divisor $(D)$:
$$\calE = \{F_{1}, E_{1,2}, E_{1,4}, E_{1,6}, E_{1,8}, F_{\infty}, E_{\infty, 2}, E_{\infty, 4}, E_{\infty, 6}, E_{\infty, 8}, E_{\omega}, E_{\omega^2}\}.$$
The singular points of four singular fibers of $g$ consist of twelve points denoted by
$\{ p_1,..., p_{12}\}$ which are the images of the twelve curves in $\calE$.
We may assume that $p_{11}, p_{12}$ are the images of $E_{\omega}, E_{\omega^2}$ respectively.  Then $p_{11}, p_{12}$ (resp. $p_1,..., p_{10}$) are the singular points of the singular fibers
of $g$ of type $\I_1$ (resp. of type $\I_5$).
Each of the twenty nonsingular rational curves passes through two points from $\{p_1,..., p_{12}\}$ because its preimage on $Y$ meets exactly two curves from twelve curves in $\calE$ (see Table \ref{Table2}). 

Let $\calS_1$ be the set of fifteen nonsingular rational curves which are ten components 
of two singular fibers of $g$ of type ${\rm I}_5$ and five 2-sections which do not pass
through $p_{11}$ and $p_{12}$, that is, the images of $s_0, s_1,..., s_4$.
Then the dual graph of the curves in $\calS_1$ is the line graph of the Petersen graph.
For the Petersen graph, see Figure \ref{petersen}.  
Here the line graph $L(G)$ of a graph $G$ is the graph 
whose vertices correspond to the edges in $G$ bijectively and two vertices in $L(G)$ are
joined by an edge iff the corresponding edges meet at a vertex in $G$.
In the following Figure \ref{enriques12}, we denote by ten dots the ten points $\{p_1,..., p_{10}\}$.  The fifteen lines denote the fifteen nonsingular rational curves in $\calS_1$.

\begin{figure}[!htb]
 \begin{center}
  \includegraphics[width=50mm]{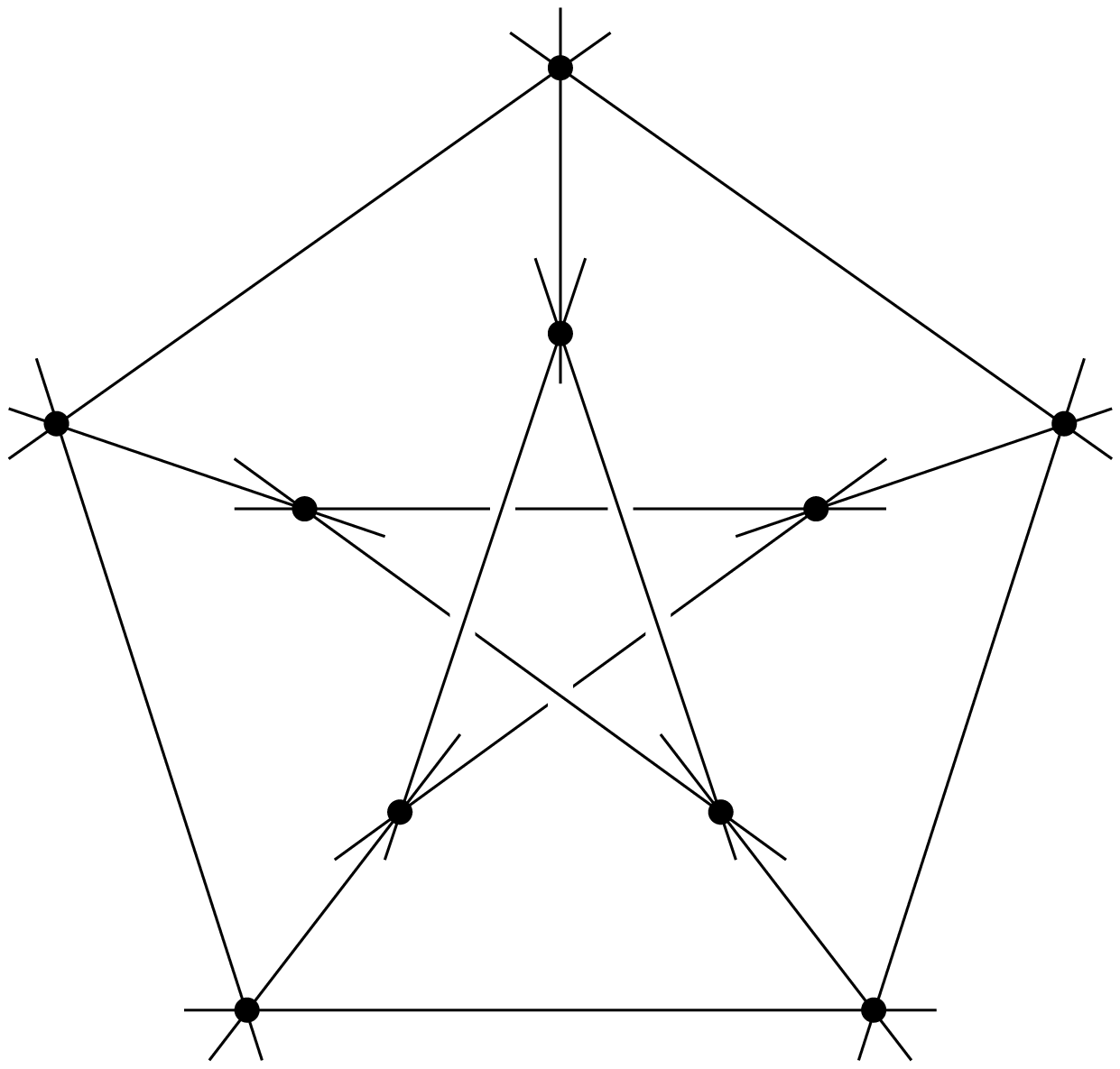}
 \end{center}
 \caption{}
 \label{enriques12}
\end{figure}

On the other hand,
let $\calS_2$ be the set of curves which are the images of $m_0,..., m_4$.
Then the dual graph of the curves in $\calS_2$ is the complete graph with
five vertices in which each pair of the vertices forms the extended Dynkin diagram of type
$\tilde{A}_1$ because all of them pass through the two points $p_{11}$ and $p_{12}$. Each vertex in $\calS_1$ meets exactly one vertex in $\calS_2$ with multiplicity 2, because
any component of the singular fibers of type $\I_{10}$ meets exactly one section from
$m_0,..., m_4$ (see Table \ref{Table2}) and $s_i$ meets only $m_i$ ($i=0,1,...,4)$ (see the equation (\ref{int-sections})).  On the other hand, the vertex in 
$\calS_2$ meets three vertices in $\calS_1$ with multiplicity 2, because $m_i$ meets
one component of each singular fiber of type $\I_{10}$ and $s_i$.
The dual graph $\Gamma$ of the twenty curves in $\calS_1$ and $\calS_2$ forms the same dual graph of
nonsingular rational curves of the Enriques surfaces of type $\VII$ given in Figure \ref{Figure7-7} (Fig. 7.7 in \cite{Ko}).

\begin{figure}[!htb]
 \begin{center}
  \includegraphics[width=60mm]{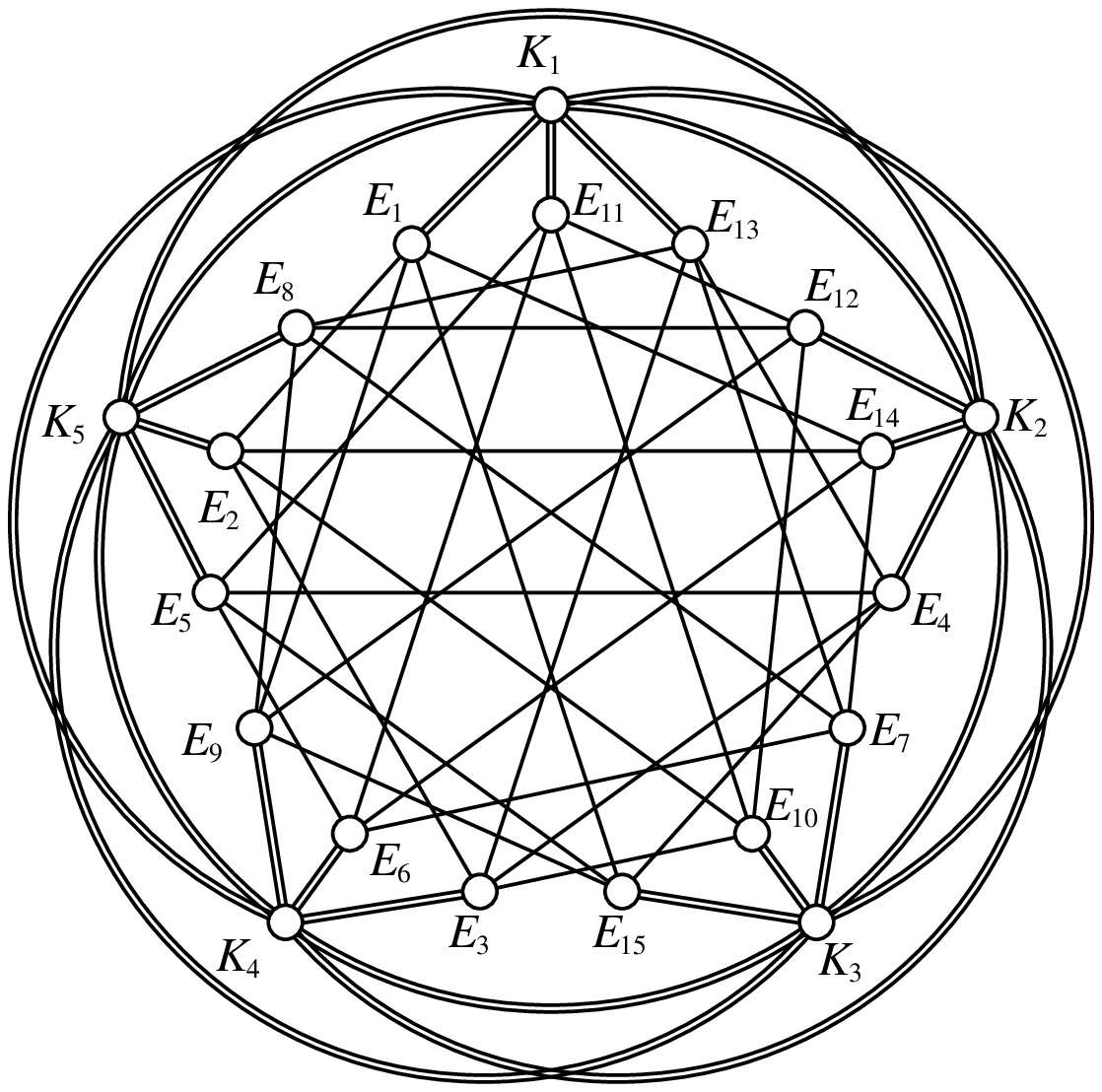}
 \end{center}
 \caption{}
 \label{Figure7-7}
\end{figure}

\noindent
The 15 curves in $\calS_1$ (resp. five curves in $\calS_2$) correspond to
$E_1, ..., E_{15}$ (resp. $K_1,..., K_5$) in Figure \ref{Figure7-7}.  It is easy to see that the maximal parabolic subdiagrams in $\Gamma$ are 
$$\tilde{A}_8,\ \tilde{A}_4\oplus \tilde{A}_4, \ \tilde{A}_5\oplus \tilde{A}_2\oplus
\tilde{A}_1, \ \tilde{A}_7\oplus \tilde{A}_1$$
which are coresponding to elliptic fibrations of type 
$$({\rm I}_9),\ ({\rm I}_5, {\rm I}_5), \ ({\rm I}_6, {\rm IV}, {\rm I}_2), \ ({\rm I}_8, {\rm III}),$$
respectively.
It follows from Vinberg's criterion (Proposition \ref{Vinberg}) that $W(X)$ is of finite index in $\O(\Num(X))$.
The same argument in \cite{Ko}, (3.7) implies that $X$ contains exactly twenty nonsingular rational curves
in $\calS_1, \calS_2$.

\begin{lemma}\label{injectiv}
The map $\rho : {\rm Aut}(X) \to {\rm O}({\rm Num}(X))$ is injective.
\end{lemma}
\begin{proof}
Let $\varphi \in \Ker(\rho)$.  Then $\varphi$ preserves each nonsingular rational curve on $X$.
Since each nonsingular rational curve meets other curves at least three points,
$\varphi$ fixes all 20 nonsingular rational curves pointwisely. Now consider the elliptic fibration $g : X \to {\bf P}^1$.  Since this fibration has has ten 2-sections, $\varphi$ fixes a general fiber of $g$ and hence $\varphi$ is identity.
\end{proof}

By Proposition \ref{finiteness}, we now have the following theorem.

\begin{theorem}\label{main2}
The automorphism group ${\rm Aut}(X)$ is isomorphic to the symmetric group $\mathfrak{S}_5$ of degree five and $X$ contains exactly twenty nonsingular rational curves whose dual graph is of type ${\rm VII}$.
\end{theorem}
\begin{proof}
We have already showed that ${\rm Aut}(X)$ is finite and $X$ contains exactly twenty 
nonsigular rational curves whose dual graph $\Gamma$ is of type ${\rm VII}$.
It follows Lemma \ref{injectiv} that ${\rm Aut}(X)$ is a subgroup of $\Aut(\Gamma) \cong\mathfrak{S}_5$.  Then by the same argument in \cite{Ko}, (3.7), we see that ${\rm Aut}(\Gamma)$ is represented by automorpisms of $X$.
\end{proof}

\begin{theorem}\label{non-isotrivial}
The one dimensional family $\{X_{a,b}\}$ is non-isotrivial.
\end{theorem}
\begin{proof}
Denote by $\Gamma$ the dual graph of all nonsingular rational curves on $X$ which is
given in Figure \ref{Figure7-7}.
$\Gamma$ contains only finitely many extended Dynkin diagrams (= the disjoint union of $\tilde{A}_m, \tilde{D}_n, \tilde{E}_k$), that is, $\tilde{A}_8, \tilde{A}_7\oplus \tilde{A}_1, \tilde{A}_4\oplus \tilde{A}_4, \tilde{A}_5\oplus \tilde{A}_2\oplus \tilde{A}_1$ (see also Kondo \cite{Ko}, page 274, Table 2).
Note that the elliptic fibrations on $X$ bijectively correspond to
the extended Dynkin diagrams in $\Gamma$. This implies that $X$ has only finitely many 
elliptic fibrations.
The $j$-invariant of the elliptic curve which appears as the fiber $E_{a}$ defined by $t = a$ 
of the elliptic fibration $f : Y \longrightarrow {\bf P}^{1}$ is equal to 
$a^{24}/(a + 1)^{10}(a^2 + a + 1)^2$ (cf. section 3). Consider
the multiple fiber $2E'_{a}$ on the elliptic fibration on the Enriques surface $X$ which is the image of $E_a$. 
Since we have a purely inseparable
morphism of degree 2 from $E_{a}$ to $E'_{a}$, we see that the $j$-invariant of $E'_{a}$
is equal to $a^{48}/(a + 1)^{20}(a^2 + a + 1)^4$.  This implies the infiniteness of the number of elliptic curves which appear as the multiple fibers of the elliptic fibration on an Enriques surface 
in our family of Enriques surfaces with parameter $a$. Therefore, in our family of 
Enriques surfaces there are infinitely many non-isomorphic ones (see  also Katsura-Kond\=o \cite{KK},
Remark 4.9).
\end{proof}

\begin{remark}
The pullback of an elliptic fibration $\pi : X\to {\bf P}^1$ to the covering
$K3$ surface $Y$ gives an elliptic fibration $\tilde{\pi} : Y\to {\bf P}^1$. The type of reducible singular fibers of $\tilde{\pi}$ is $({\rm I}_{10}, {\rm I}_{10}, {\rm I}_2, {\rm I_2})$ if $\pi$ is of type $\tilde{A}_4\oplus \tilde{A}_4$, $({\rm I}_{16}, {\rm I}_1^*)$
if $\pi$ is of type $\tilde{A}_7\oplus \tilde{A}_1$, $({\rm I}_{12}, {\rm III}^*, {\rm I}_4)$ if $\pi$ is of type $\tilde{A}_5\oplus \tilde{A}_2\oplus \tilde{A}_1$, and 
type $({\rm I}_{18}, {\rm I}_2, {\rm I}_{2}, {\rm I}_2)$ if $\pi$ is of type $\tilde{A}_8$, respectively.
\end{remark}

The following theorem is due to M. Sch\"utt and H. Ito.
\begin{theorem}\label{non-existVII}
There are no singular Enriques surfaces with the dual graph of type ${\rm VII}$.  
\end{theorem}
\begin{proof}
Assume that there exists an Enriques surface $S$ with the dual graph of type ${\rm VII}$.  
In the dual graph of type {\rm VII} there exists a parabolic subdiagram $\tilde{A}_5\oplus \tilde{A}_2 \oplus \tilde{A}_1$.  By Proposition \ref{singular-fiber}, it corresponds to an elliptic fibration on $S$ with singular fibers of type $({\rm IV}, {\rm I}_2, {\rm I}_6)$.  For example, the linear system $|2(E_1+E_2+E_{14})|$ defines a such fibration.  Moreover the dual graph of type ${\rm VII}$ tells us that the singular fiber $E_1+E_2+E_{14}$ of type ${\rm IV}$ is a multiple fiber because $E_3$ is a 2-section of this fibration (see Figure \ref{Figure7-7}).  This contradicts to Proposition \ref{multi-fiber}, (ii).
\end{proof}

\section{Examples of singular $K3$ surfaces with a finite automorphism group}\label{sec4}

\subsection{Type $\I$}\label{type1}
Let $(x_0,x_1,x_2,x_3)$ be a homogeneous coodinate of ${\bf P}^3$.
Consider the nonsingular quadric $Q$ in ${\bf P}^3$ defined by
\begin{equation}\label{type1quadric}
x_0x_3 + x_1x_2=0
\end{equation}
which is the image of the map
$${\bf P}^1\times {\bf P}^1 \to {\bf P}^3,\quad ((u_0,u_1),(v_0,v_1)) \to (u_0v_0, u_0v_1, u_1v_0, u_1v_1).$$
The involution of ${\bf P}^1\times {\bf P}^1$ 
$$((u_0,u_1),(v_0,v_1)) \to ((u_1,u_0),(v_1,v_0))$$
induces an involution
\begin{equation}\label{typeIinv}
\tau : (x_0,x_1,x_2,x_3) \to (x_3,x_2,x_1,x_0)
\end{equation}
of $Q$ whose fixed point set on $Q$ is one point $(1,1,1,1)$.
Consider four lines on $Q$ defined by
$$L_{01}:x_0=x_1=0,\quad L_{02}: x_0=x_2=0,$$  
$$L_{13}: x_1=x_3=0, \quad L_{23}: x_2=x_3=0,$$
and a $\tau$-invariant pencil of quadrics
$$C_{\lambda,\mu} : \lambda (x_0+x_3)(x_1+x_2)+ \mu x_0x_3 =0$$
passing through the four vertices 
$$(1,0,0,0), \quad (0,1,0,0),\quad (0,0,1,0),\quad (0,0,0,1)$$ 
of the quadrangle $L_{01}, L_{02}, L_{13}, L_{23}$.  
Note that two conics
$$Q_1: x_0+x_3=0, \quad Q_2: x_1+x_2=0$$
tangent to $C_{\lambda,\mu}$ at two vertices of the quadrangle.  
Obviously 
$$C_{1,0} = Q_1 +Q_2, \quad C_{0,1} = L_{01}+L_{02} + L_{13}+L_{23},$$ 
and 
$C_{\lambda,\mu}$ $(\lambda\cdot \mu\not=0)$ is a nonsingular elliptic curve.
Thus we have the same configuration of
curves given in \cite{Ko}, Figure 1.1 except $Q_1$ and $Q_2$ tangent at $(1,1,1,1)$.

Now we fix $(\lambda_0, \mu_0)\in {\bf P}^1$ $(\lambda_0\cdot \mu_0\not=0)$ and
take Artin-Schreier covering $S \to Q$ defined by the triple
$(L, a, b)$ where $L= \calO_Q(2,2)$, $a \in H^0(Q,L)$ and $b\in H^0(Q,L^{\otimes 2})$
satisfying $Z(a) = C_{0,1}$ and 
$Z(b) = C_{0,1} + C_{\lambda_0,\mu_0}$.
The surface $S$ has four singular points over the four vertices of quadrangle given
locally by $z^2 +uvz + uv(u+v)=0$.  In the notation in Artin's list (see \cite{A}, \S 3), it is of type $D^1_4$.  Let $Y$ be the minimal nonsingular model of $S$.
Then the exceptional divisor over a singular point has the dual graph of type $D_4$.
The canonical bundle formula implies that $Y$ is a $K3$ surface.
The pencil $\{C_{\lambda,\mu}\}_{(\lambda,\mu)\in {\bf P}^1}$ induces an elliptic fibration on $Y$.
The preimage of $L_{01}+L_{02} + L_{13}+L_{23}$ is the singular fiber of type $\I_{16}$
and the preimage of $Q_1+Q_2$ is the union of two singular fibers of type $\III$.
Note that the pencil has four sections.  Thus we have 24 nodal curves on $Y$.
Note that the dual graph of these 24 nodal curves coincide with the one given in
\cite{Ko}, Figure 1.3. 
The involution $\tau$
can be lifted to a fixed point free involution $\sigma$ of $Y$ because the branch divisor $C_{0,1}$ does not contain the point $(1,1,1,1)$.  By taking the quotient
of $Y$ by $\sigma$, we have a singular Enriques surface $X=Y/\langle \sigma \rangle$.
The above elliptic fibration induces an elliptic pencil on $X$ with singular fibers of type $\I_8$ and of type $\III$.   
Since the ramification divisor of the covering $S\to Q$ is the preimage of $L_{01}+L_{02} + L_{13}+L_{23}$, the multiple fiber of this pencil is the singular fiber of type $\I_8$.
By construction, $X$ contains twelve nonsingular rational curves whose dual graph coincides with the one given in \cite{Ko}, Figure 1.4. 
It follows from Vinberg's criterion (Proposition \ref{Vinberg}) that $W(X)$ is of finite index in $\O(\Num(X))$, and hence
the automorphism group $\Aut(X)$ is finite (Proposition \ref{finiteness}). The same argument as in the proof of \cite{Ko}, Theorem 3.1.1 shows that $\Aut(X)$ is isomorphic to the digedral group $D_4$ of order 8.  Thus we have the following theorem.

\begin{theorem}\label{Ithm}
These $X$ form a one dimensional family of singular Enriques surfaces whose dual graph of nonsingular rational curves is of type ${\rm I}$.  The automorphism group ${\rm Aut}(X)$ is isomorphic to the dihedral group $D_4$ of order $8$.
\end{theorem}

\begin{theorem}\label{non-existI}
There are no
classical and supersingular Enriques surfaces with the dual graph of type ${\rm I}$. 
\end{theorem}
\begin{proof}
From the dual graph of type $\I$, we can see that such Enriques surface has an elliptic fibration with
a multiple fiber of type $\I_8$.  The assersion now follows from Proposition \ref{multi-fiber}.
\end{proof}

\begin{remark}\label{typeInumtrivial}
In the above, we consider special quadrics $C_{\lambda, \mu}$ tangent to $Q_1, Q_2$.
If we drop this condition and consider general $\tau$-invariant quadrics through
the four vertices of the quadrangle $L_{01}, L_{02}, L_{13}, L_{23}$, we have a two dimensional family of singular Enriques surfaces $X$.  The covering transformation of $Y \to S$ descends to a numerically trivial involution of $X$, that is, an involution of $X$ acting trivially on $\Num(X)$. In the appendix \ref{kummer}, we discuss Enriques surfaces with a numerically trivial involution.
\end{remark}

\subsection{Type $\II$}\label{type2}
We use the same notation as in (\ref{type1}).
We consider a $\tau$-invariant pencil of quadrics defined by
$$C_{\lambda,\mu} : \lambda (x_0+x_1+x_2+x_3)^2+ \mu x_0x_3 =0 
$$
which tangents to the quadrangle $L_{01}, L_{02}, L_{13}, L_{23}$
at $(0,0,1,1)$, $(0,1,0,1)$, $(1,0,1,0)$, $(1,1,0,0)$ respectively.
Let
$$L_1: x_0+x_1=x_2+x_3=0,\quad L_2: x_0+x_2=x_1+x_3=0$$
be two lines on $Q$ which passes the tangent points of $C_{\lambda,\mu}$ and
the quadrangle $L_{03}, L_{12}, L_{02}, L_{13}$.
Note that 
$$C_{1,0}= 2L_1+ 2L_2, \quad C_{0,1}= L_{01}+L_{02} + L_{13}+L_{23},$$ 
and 
$C_{\lambda,\mu}$ $(\lambda\cdot \mu\not=0)$ is a nonsingular elliptic curve.
Thus we have the same configuration of
curves given in \cite{Ko}, Figure 2.1.

Now we fix $(\lambda_0, \mu_0)\in {\bf P}^1$ $(\lambda_0 \cdot \mu_0 \not=0)$ and
take Artin-Schreier covering $S \to Q$ defined by the triple
$(L, a, b)$ where $L= \calO_Q(2,2)$, $a \in H^0(Q,L)$ and $b\in H^0(Q,L^{\otimes 2})$
satisfying $Z(a) = C_{0,1}$ and 
$Z(b) = C_{0,1} + C_{\lambda_0,\mu_0}$.
The surface $S$ has four singular points over the four tangent points 
of $C_{\lambda_0,\mu_0}$ with the quadrangle and 
four singular points over the four vertices of the quadrangle.  
A local equation of each of the first four singular points
is given by $z^2 +uz + u(u+v^2)=0$ and the second one is given by $z^2 + uvz + uv=0$. 
In the first case, by the change of coordinates
$$
t=z +\omega u + v^2,\quad s = z +\omega^2 u + v^2,\quad v = v
$$ 
($\omega^3=1$, $\omega\not= 1$), then
we have $v^4 +ts=0$ which gives a rational double point of type $A_3$.  In the second case, obviously, it is a rational double point of type $A_1$.
Let $Y$ be the minimal nonsingular model of $S$.
Then the exceptional divisor over a singular point in the first case has the dual graph of type $A_3$ and in the second case the dual graph of type $A_1$.
The canonical bundle formula implies that $Y$ is a $K3$ surface.
The pencil $\{C_{\lambda,\mu}\}_{(\lambda,\mu)\in {\bf P}^1}$ induces an elliptic fibration on $Y$.
The preimage of $L_{01}+L_{02} + L_{13}+L_{23}$ is the singular fiber of type $\I_{8}$
and the preimage of $C_{1,0}$ is the union of two singular fibers of type $\I_1^*$.
Note that the pencil has four sections.  Thus we have 24 nodal curves on $Y$.
Note that the dual graph of these 24 nodal curves coincide with the one given in
\cite{Ko}, Figure 2.3. 
The involution $\tau$
can be lifted to a fixed point free involution $\sigma$ of $Y$ because the branch divisor $C_{0,1}$ does not contain the point $(1,1,1,1)$.  By taking the quotient
of $Y$ by $\sigma$, we have a singular Enriques surface $X=Y/\langle \sigma \rangle$.
The above elliptic fibration induces an elliptic pencil on $X$ with singular fibers of type $\I_4$ and of type $\I_1^*$.   
Since the ramification divisor of the covering $S\to Q$ is the preimage of $L_{01}+L_{02} + L_{13}+L_{23}$, the multiple fiber of this pencil is the singular fiber of type $\I_4$.
By construction, $X$ contains twelve nonsingular rational curves whose dual graph $\Gamma$ coincides with the one
given in \cite{Ko}, Figure 2.4.  The same argument as in the proof of \cite{Ko}, Theorem 3.2.1 shows that $W(X)$ is of finite index in $\O(\Num(X))$ and $X$ contains only these twelve nonsingular rational curves.  It now follows from Proposition \ref{finiteness} that
the automorphism group $\Aut(X)$ is finite.  
By the similar argument as in the proof of Lemma \ref{injectiv}, we see that 
the map $\rho : \Aut(X) \to \O(\Num(X))$ is injective. Moreover, by the same argument 
as in the proof of \cite{Ko}, Theorem 3.2.1, $\Aut(X)$ is isomorphic to $\Aut(\Gamma) \cong \mathfrak{S}_4$.
Thus we have the following theorem.

\begin{theorem}\label{IIthm}
These $X$ form a one dimensional family of singular Enriques surfaces whose dual graph of nonsingular rational curves is of type ${\rm II}$.  The automorphism group ${\rm Aut}(X)$ is isomorphic to the symmetric group $\mathfrak{S}_4$ of degree four.
\end{theorem}

\begin{theorem}\label{non-existII}
There are no
classical and supersingular Enriques surfaces with the dual graph of type ${\rm II}$. 
\end{theorem}
\begin{proof}
From the dual graph of type $\II$, we can see that such Enriques surface has an elliptic fibration with
a multiple fiber of type $\I_4$.  The assersion now follows from Proposition \ref{multi-fiber}.
\end{proof}

\subsection{Type $\VI$}\label{type6}
Over the field of complex numbers, the following example was studied 
by Dardanelli and van Geemen \cite{DvG}, Remark 2.4.  This surface $X$ is isomorphic to 
the Enriques surface of type $\VI$ given in \cite{Ko} (In \cite{DvG}, Remark 2.4, they claimed that $X$ is of type $\IV$, but this is a misprint).  Their construction 
works well in characteristic 2.

Let $(x_1,\cdots , x_5)$ be a homogeneous coodinate of ${\bf P}^4$.
Consider the surface $S$ in ${\bf P}^4$ defined by

\begin{equation}
\sum_{i=1}^{5} x_i = \sum_{i=1}^5 {1/x_i} = 0.
\end{equation}

\noindent
Let 
$$\ell_{ij} : x_i=x_j=0 \ \ (1\leq i<j \leq 5),$$ 
$$p_{ijk} : x_i=x_j=x_k=0 \ \ (1\leq i<j<k\leq 5).$$
The ten lines $\ell_{ij}$ and ten points $p_{ijk}$ lie on $S$.
By taking partial derivatives, we see that $S$ has ten nodes at $p_{ijk}$.  Let $Y$ be the minimal
nonsingular model of $S$.  Then $Y$ is a $K3$ surface. 
Denote by $L_{ij}$ the proper transform of $\ell_{ij}$ and by $E_{ijk}$ the exceptional curve over $p_{ijk}$.  The Cremonat transformation
$$(x_i) \to \left({1/x_i}\right)$$
acts on $Y$ as an automorphism $\sigma$ of order 2.  Note that the fixed point set of
the Cremonat transformation is 
exactly one point $(1,1,1,1,1)$.  Hence $\sigma$ is a fixed point free involution of $Y$.
The quotient surface $X=Y/\langle \sigma \rangle$ is a singular Enriques surface. 
Obviously the permutation group $\mathfrak{S}_5$ acts on $S$ which commutes with $\sigma$.  Therefore $\mathfrak{S}_5$ acts on $X$ as automorphisms. The involution $\sigma$ changes $L_{ij}$ and $E_{klm}$, where $\{i,j,k,l,m\} =\{1,2,3,4,5\}$.
The images of twenty nonsingular rational curves $L_{ij}$, $E_{ijk}$ give ten nonsingular rational curves on $X$ whose
dual graph is given by the following Figure \ref{petersen}. Note that this graph is 
well known called the Petersen graph.

\begin{figure}[htbp]
 \begin{center}
  \includegraphics[width=60mm]{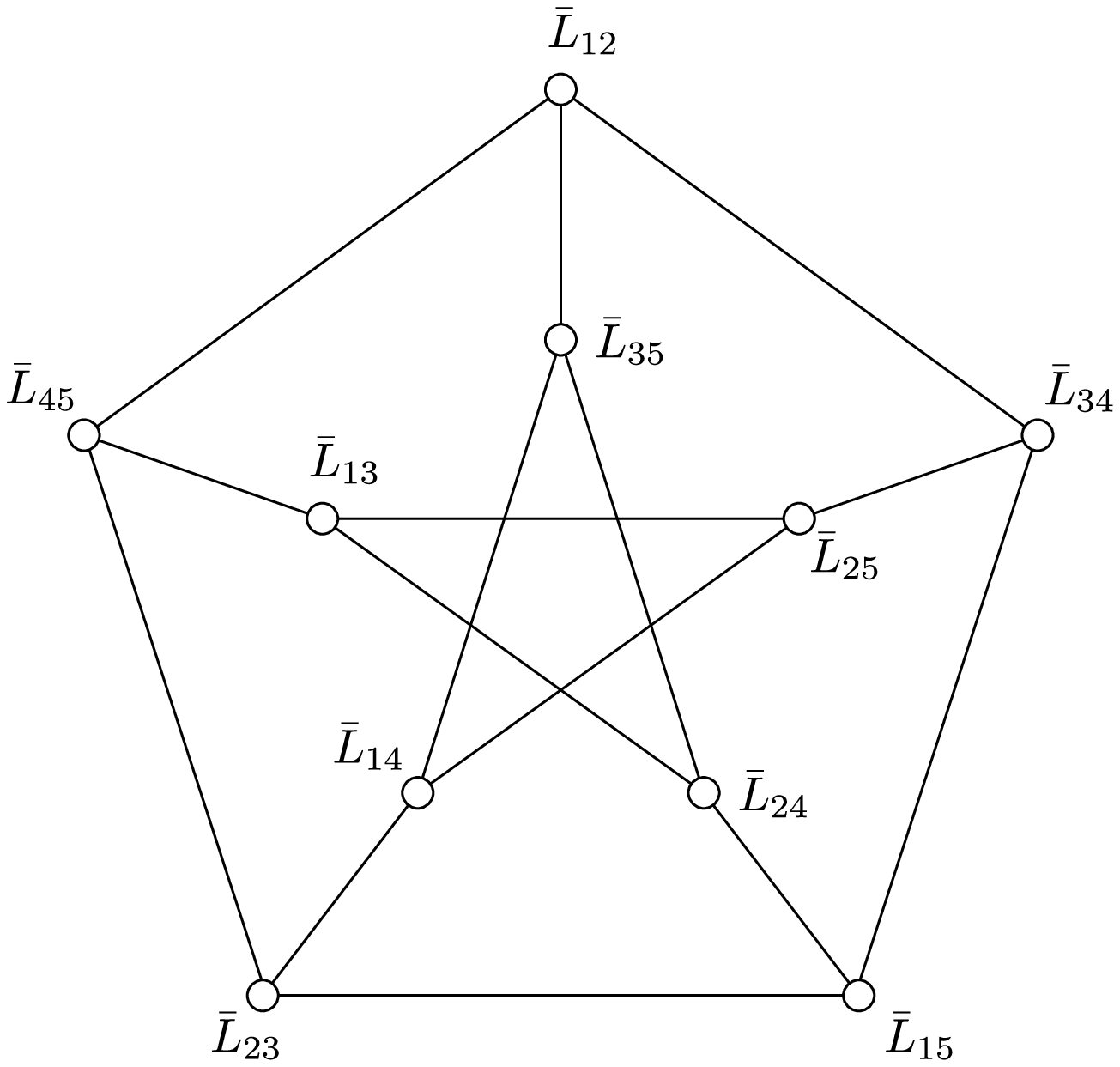}
 \end{center}
 \caption{}
 \label{petersen}
\end{figure}
\noindent
Here $\bar{L}_{ij}$ is the image of $L_{ij}$ (and $E_{klm}$).
Note that $\mathfrak{S}_5$ is the automorphism group of the Petersen graph.

The hyperplane section $x_i+x_j=0$ on $S$ is the union of the double line $2\ell_{ij}$ and two lines through $p_{klm}$ defined by $x_kx_l+x_kx_m+x_lx_m=0$.  Thus we have
additional twenty nodal curves on $Y$.  
Note that the Cremonat transformation changes
two lines defined by $x_kx_l+x_kx_m+x_lx_m=0$.  
Thus $X$ contains twenty nonsingular rational curves
whose dual graph $\Gamma$ coincides with the one
of the Enriques surface of type {\rm VI} (see Fig.6.4 in \cite{Ko}).  
It now follows from Proposition  \ref{finiteness} that
the automorphism group $\Aut(X)$ is finite.  
The same argument as in the proof of \cite{Ko}, Theorem 3.1.1 shows 
that $X$ contains 
only these 20 nonsingular rational curves.  
By a similar argument to the one in the proof of Lemma \ref{injectiv}, we see that 
the map $\rho : \Aut(X) \to \O(\Num(X))$ is injective. 
Since the classes of twenty nonsingular rational curves generate $\Num(X)\otimes {\bf Q}$,  
$\Aut(X)$ is isomorphic to $\Aut(\Gamma) \cong \mathfrak{S}_5$.
Thus we have the following theorem.

\begin{theorem}\label{VIthm}
The surface $X$ is a singular Enriques surfaces whose dual graph of nonsingular rational curves is of type ${\rm VI}$.  The automorphism group $\Aut(X)$ is isomorphic to the symmetric group $\mathfrak{S}_5$ of degree five.
\end{theorem}

\begin{theorem}\label{non-existVI}
There are no
classical and supersingular Enriques surfaces with the dual graph of type ${\rm VI}$. 
\end{theorem}
\begin{proof}
A pentagon in the Figure \ref{petersen}, for example, 
$|\bar{L}_{12} + \bar{L}_{34}+ \bar{L}_{15}+ \bar{L}_{24}+\bar{L}_{35}|$, 
defines an elliptic fibration on $X$.
The multiple fiber of this fibration is nothing but the pentagon, that is, of type 
$\I_5$.
The assertion now follows from Proposition \ref{multi-fiber}.
\end{proof}

\begin{remark}\label{type7}
Over the field of complex numbers, Ohashi found that the Enriques surface 
of type $\VII$ in \cite{Ko} is
isomorphic to the following surface (see \cite{MO}, \S 1.2). 
Let $(x_1,\cdots , x_5)$ be homogeneous coodinates of ${\bf P}^4$.
Consider the surface in ${\bf P}^4$ defined by

\begin{equation}
\sum_{i< j} x_i x_j = \sum_{i<j<k} x_ix_jx_k = 0
\end{equation}
which has five nodes at coodinate points and whose minimal resolution is a $K3$ surface $Y$.
The standard Cremonat transformation
$$(x_i) \to \left({1/ x_i}\right)$$
acts on $Y$ as a fixed point free involution $\sigma$.  Thus 
the quotient surface $X=Y/\langle \sigma \rangle$ is a complex Enriques surface.  
In characteristic 2, the involution $\sigma$ has a fixed point $(1,1,1,1,1)$ on $Y$, and
hence the quotient is not an Enriques surface.
\end{remark}

\subsection{Type $\III, \IV, \V$}\label{type345}
In each case of type $\III$, $\IV$, $\V$, from the dual graph (cf. Kondo \cite{Ko},  Figures 3.5, 4.4, 5.5)
we can find an elliptic fibration
which has two reducible multiples fibers.  In fact, the parabolic subdiagram
of type $\tilde{D}_6 \oplus \tilde{A}_1 \oplus \tilde{A}_1$ in case $\III$ 
(of type $\tilde{A}_3 \oplus \tilde{A}_3 \oplus \tilde{A}_1 \oplus \tilde{A}_1$ in case $\IV$, of type $\tilde{A}_5 \oplus \tilde{A}_2 \oplus \tilde{A}_1$ in case $\V$) defines such an elliptic fibration 
(see \cite{Ko}, Table 2, page 274). Hence if an Enriques
surface with the same dual graph of nodal curves exists in characteristic 2,
then it should be classical (Proposition \ref{multi-fiber}).  On the other hand, in each case of type $\III$, $\IV$, $\V$,
there exists an elliptic fibration which has a reducible multiple fiber 
of multiplicative type  (see \cite{Ko}, Table 2, page 274).  
However this is impossible 
because any multiple fiber of an elliptic fibration on a classical Enriques surface
is nonsingular or singular of additive type (Proposition \ref{multi-fiber}).  
Thus we have prove the following theorem.

\begin{theorem}\label{non-existIII}
There are no Enriques surfaces with the same dual graph as
in case of type ${\rm III}$, ${\rm IV}$ or ${\rm V}$.  
\end{theorem}

Combining Theorems \ref{main2}, \ref{non-existVII}, \ref{Ithm}, \ref{non-existI}, \ref{IIthm}, \ref{non-existII}, \ref{VIthm}, \ref{non-existVI}, \ref{non-existIII},
we have the Table \ref{Table1} in the introduction.

\begin{remark}\label{extra}
In characteristic 2, there exist Enriques surfaces with a finite group of automorphisms
whose dual graphs of all nonsingular rational curves do not appear 
in the case of complex surfaces.  For example, 
it is known that there exists an Enriques surface $X$ which has a genus 1 fibration
with a multiple singular fiber of type $\tilde{E}_8$ and with a 2-section 
(Ekedahl and Shepherd-Barron\cite{ES}, Theorem A, Salomonsson\cite{Sa}, Theorem 1).
We have ten nonsingular rational curves on $X$, that is, nine components of the singular fiber and
a 2-section, whose dual graph is given in Figure \ref{E10Dynkin}.

\begin{figure}[htbp]
\begin{center}
\begin{picture}(120,30)
\put(0, 20){\circle{5}}
\put(0, 30){\makebox(0, 0){}}
\put(2, 20){\line(1, 0){15}}
\put(20, 20){\circle{5}}
\put(20, 30){\makebox(0, 0){}}
\put(22, 20){\line(1, 0){15}}
\put(40, 20){\circle{5}}
\put(40, 30){\makebox(0, 0){}}
\put(42, 20){\line(1, 0){15}}
\put(60, 20){\circle{5}}
\put(60, 30){\makebox(0, 0){}}
\put(62, 20){\line(1, 0){15}}
\put(80, 20){\circle{5}}
\put(80, 30){\makebox(0, 0){}}
\put(82, 20){\line(1, 0){15}}
\put(100, 20){\circle{5}}
\put(100, 30){\makebox(0, 0){}}
\put(102, 20){\line(1, 0){15}}
\put(120, 20){\circle{5}}
\put(120, 30){\makebox(0, 0){}}
\put(122, 20){\line(1, 0){15}}
\put(140, 20){\circle{5}}
\put(140, 30){\makebox(0, 0){}}
\put(142, 20){\line(1, 0){15}}
\put(160, 20){\circle{5}}
\put(160, 30){\makebox(0, 0){}}
\put(40, 17){\line(0, -1){15}}
\put(40, 0){\circle{5}}
\put(50, 0){\makebox(0, 0){}}
\put(120, 0){\makebox(0, 0){}}
\end{picture}
 \caption{}
 \label{E10Dynkin}
\end{center}
\end{figure}

\noindent
It is easy to see that they generate $\Num(X) \cong U\oplus E_8$.  Moreover it is known that the reflection subgroup generated by reflections associated with these $(-2)$-vectors is of
finite index in $\O(\Num(X))$ (Vinberg \cite{V}, Table 4; also see Proposition \ref{Vinberg}) and hence $\Aut(X)$ is finite (Proposition \ref{finiteness}).
\end{remark}

\appendix
\section{The height of the covering $K3$ surfaces of singular Enriques surfaces}\label{height}

In this section we prove the following theorem.

\begin{theorem}\label{height}
In characteristic $2$, if $K3$ surface $Y$ has a fixed point free
involution, then the height $h(Y)$ of the formal Brauer group of $Y$
is equal to $1$. 
\end{theorem}

\begin{corollary}\label{height2}
Let $Y$ be the covering $K3$ surface of a singular Enriques surface.  Then the height
$h(Y) = 1$.  
\end{corollary}


\begin{proof}
Suppose $h = h(Y) \neq 1$. 
Since ${\rm H}^2(Y, {\calO}_{Y})$ is the tangent space of the formal Brauer group of $Y$
(cf. Artin-Mazur \cite{AM}, Corollary (2.4)), the Frobenius map 
$$
F : {\rm H}^2(Y, {\calO}_{Y}) \rightarrow {\rm H}^2(Y, {\calO}_{Y})
$$
is a zero map. Then, we have an isomorphism
$$
{\rm id} - F : {\rm H}^2(Y, {\calO}_{Y}) \rightarrow {\rm H}^2(Y, {\calO}_{Y}).
$$
Let $W_{i}({\calO}_{Y})$ be the sheaf of ring of Witt vectors of length $i$ on $Y$.
Assume ${\rm id} - F : {\rm H}^2(Y, W_{i-1}({\calO}_{Y})) \longrightarrow 
{\rm H}^2(Y, W_{i-1}({\calO}_{Y}))$ is an isomorphism.
We have an exact sequence
$$
0    \rightarrow   W_{i-1}({\calO}_{Y}) \stackrel{V}{\longrightarrow} W_{i}({\calO}_{Y}) \stackrel{R}{\longrightarrow} {\calO}_{Y}  \rightarrow
 0,
$$
where $V$ is the Verschiebung and $R$ is the restriction. Then, we have a diagram
{\small
$$
\begin{array}{ccccccccc}
0    & \rightarrow &  {\rm H}^2(Y, W_{i-1}({\calO}_{Y}))  & \stackrel{V}{\longrightarrow}& {\rm H}^2(Y, W_{i}({\calO}_{Y})) &\stackrel{R}{\longrightarrow} &  {\rm H}^2(Y, {\calO}_{Y}) & \rightarrow
& 0\\
  &   & {\rm id} - F \downarrow &   & {\rm id} - F \downarrow &  & {\rm id} - F \downarrow & & \\
0    & \rightarrow &  {\rm H}^2(Y, W_{i-1}({\calO}_{Y}))  & \stackrel{V}{\longrightarrow}& {\rm H}^2(Y, W_{i}({\calO}_{Y})) &\stackrel{R}{\longrightarrow} &  {\rm H}^2(Y, {\calO}_{Y}) & \rightarrow
& 0.
\end{array}
$$
}
By the assumption of induction, the first and the third downarrows are isomorphisms. 
Therefore, by the 5-lemma, we have an isomorphism
$$
{\rm id} - F : {\rm H}^2(Y,  W_{i}({\calO}_{Y})) \cong {\rm H}^2(Y,  W_i({\calO}_{Y})).
$$
Therefore, taking the projective limit, we have an isomorphism
$$
{\rm id} - F : {\rm H}^2(Y,  W({\calO}_{Y})) \cong {\rm H}^2(Y,  W({\calO}_{Y}))
$$
Therefore, denoting by $K$ the quotient field of the ring of Witt vectors $W(k)$ of infinite length,
we have an isomorphism
$$
{\rm id} - F : {\rm H}^2(Y,  W({\calO}_{Y}))\otimes K \cong {\rm H}^2(Y,  W({\calO}_{Y}))\otimes K.
$$
Let ${\rm H}_{et}^2(Y, {\bf Q}_2)$ be the second 2-adic \'etale cohomology of $Y$. Then,
we have an exact sequence
$$
0 \rightarrow {\rm H}_{et}^2(Y, {\bf Q}_2)\rightarrow {\rm H}^2(Y,  W({\calO}_{Y}))\otimes K \stackrel{{\rm id} - F}{\longrightarrow} {\rm H}^2(Y,  W({\calO}_{Y}))\otimes K \rightarrow 0
$$
(cf. Crew \cite{Cr}, (2.1.2) for instance). Therefore, we have ${\rm H}_{et}^2(Y, {\bf Q}_2)= 0$.
On the other hand, we consider the quotient surface $X$ of $Y$ 
by the fixed point free involution. Then, $X$ is a singular Enriques surface
and under the assumption Crew showed $\dim {\rm H}_{et}^2(Y, {\bf Q}_2)= 1$ 
for the $K3$ covering Y of $X$ (Crew \cite{Cr}, p41), a contradiction.
\end{proof}

\section{Enriques surfaces associated with Kummer surfaces}\label{kummer}

In this section we show that Enriques surfaces given in Remark \ref{typeInumtrivial} are
obtained from Kummer surfaces associated with the product of two ordinary elliptic curves.
Let $E, F$ be two ordinary elliptic curves and let
$\iota=\iota_E\times \iota_F$ be the inversion of the abelian surface $E\times F$.
Let $\Km(E\times F)$ be the minimal resolution of the quotient surface
$(E\times F)/\langle\iota\rangle$.  It is known that $\Km(E\times F)$ is a $K3$ surface called Kummer surface associated with $E\times F$ (Shioda \cite{Shi}, Proposition 1, see also Katsura \cite{Ka}, Theorem B). The projection from $E\times F$ to $E$ gives an elliptic fibration which has
two singular fibers of type $\I^*_4$ and two sections.

Let $a \in E, \ b\in F$ be the unique non-zero 2-torsion points on $E$, $F$ respectively.  Denote by $t$ the translation of $E\times F$ by the 2-torsion point $(a,b)$.
The involution $(\iota_E\times 1_F)\circ t = t \circ (\iota_E\times 1_F)$ induces a fixed
point free involution $\sigma$ of $\Km(E\times F)$.  Thus we have an Enriques surface
$S = \Km(E\times F)/\langle\sigma\rangle$.  The involution $\iota_E\times 1_F$ (or $t$) induces a numerically trivial involution $\eta$ of $S$.  

\begin{theorem}\label{nt}
The pair $(S, \eta)$ is isomorphic to an Enriques surface given in Remark {\rm \ref{typeInumtrivial}}.
\end{theorem}
\begin{proof}
Let
$$E: y^2+xy =x^3 +bx, \quad F: y'^2+x'y'=x'^3+b'x' \quad (b, b' \in k, bb'\not=0)$$
be two ordinary elliptic curves.  The inversion $\iota_E$ is then expressed by
$$(x,y) \to (x,y+x)$$
and the translation by the non-zero 2-torsion on $E$ is given by
$$(x,y) \to (b/x, by/x^2 + b/x).$$
Then the function field of $(E\times F)/\langle \iota \rangle$ is given by
$$k((E\times F)/\langle \iota\rangle) = k(x,x', z)$$
with the relation
\begin{equation}\label{shiodakummer}
z^2 + xx'z= x^2(x'^3+b'x')+x'^2(x^3+bx)
\end{equation}
where $z=xy'+x'y$ (see Shioda \cite{Shi}, the equation (8)). The fixed point free involution $\sigma$ 
is expressed by
\begin{equation}\label{enriques-inv}
\sigma (x,x',z) = (b/x, b'/x', bb'z/x^2x'^2 + bb'/xx'),
\end{equation}
and the involution induced by $\iota_E\times 1_F$ on $\Km(E\times F)$ is given by
\begin{equation}\label{num-tri-inv}
(x,x',z)\to (x,x', z+xx').
\end{equation}

On the other hand, we consider the quadric $Q$ given in (\ref{type1quadric}).
Instead of $\tau$ in (\ref{typeIinv}), we consider the involution given by
\begin{equation}
\tau' : (x_0,x_1,x_2,x_3) \to (x_3,b'x_2,bx_1,bb'x_0)
\end{equation}
whose fixed point is $(1,b',b,bb')$.  
The Artin-Schreier covering
is defined by the equation
$$z^2 + x_0x_3z = x_0x_3(x_1x_3 +b'x_0x_2+x_2x_3+bx_0x_1)$$
(in the example given in the subsection \ref{type1}, the term $\mu(x_0x_3)^2$ appears in the Artin-Schreier covering.
If $\mu\not=0$, then changing $z$ by $z+ax_0x_3$ where $a^2+a+\mu=0$, we can delete this term).
Now, by putting here 
$$x_0=u_0v_0, \ x_1=u_0v_1, \ x_2=u_1v_0, \ x_3=u_1v_1$$
and considering an affine locus $u_0\not=0, v_0\not=0$, we have
$$z^2 + u_1v_1z=u_1v_1(u_1v_1^2+ u_1^2v_1+ bv_1+b'u_1)$$
which is the same as the equation given in (\ref{shiodakummer}).
Moreover the lifting of $\tau'$ and the covering transformation of the Artin-Schreier covering coincide
with the ones given in (\ref{enriques-inv}) and (\ref{num-tri-inv}) respectively.
\end{proof}

\begin{remark}
Using appendix A, we see that the height of the formal Brauer group of Kummer surfaces associated with the product of two ordinary elliptic curves is equal to 1.
\end{remark}
\begin{remark}
All complex Enriques surfaces with cohomologically or numerically
trivial automorphisms are classified by Mukai and Namikawa \cite{MN}, Main theorem (0.1), and Mukai \cite{M}, Theorem 3.
There are three types: one of them is an Enriques surface associated with
$\Km(E\times F)$ and the second one is mentioned in Remark \ref{typeInumtrivial}. For the third one we refer the reader to Mukai \cite{M}, Theorem 3.  
In positive characteristic, Dolgachev (\cite{D2}, Theorems 4 and 6) determined the order of cohomologically or numerically trivial automorphisms.  However, the explicit classification is not known.  The above Theorem \ref{nt} implies that
two different type of complex Enriques surfaces with a numerically trivial involution coincide in characteristic 2.
\end{remark}

\end{document}